\documentclass[leqno,11pt]{amsart} 
\usepackage{amscd,amsmath,amsfonts, amssymb, color}
\theoremstyle{plain} 
\newtheorem{theorem}{\indent\sc Theorem}[section]
\newtheorem{lemma}[theorem]{\indent\sc Lemma}
\newtheorem{corollary}[theorem]{\indent\sc Corollary}
\newtheorem{proposition}[theorem]{\indent\sc Proposition}

\theoremstyle{definition} 
\newtheorem{definition}[theorem]{\indent\sc Definition}
\newtheorem{remark}[theorem]{\indent\sc Remark}

\DeclareMathOperator{\vol}{vol}%
\DeclareMathOperator{\Div}{div}%
\DeclareMathOperator{\Hess}{Hess}%
\DeclareMathOperator{\Ric}{Ric}%

\begin{document}

\title[The connectivity of a manifold supporting an $L^{q,p}$-Sobolev inequality]{The connectivity at infinity of a manifold and  $L^{q,p}$-Sobolev inequalities}

\author{Stefano Pigola}
\author{Alberto G. Setti}
\author{Marc Troyanov}

\thanks{Work partially supported by the Italian GNAMPA and the Swiss NSF
}
\address{Dipartimento di Scienza e Alta Tecnologia - Sezione di Matematica\endgraf
Universit\`a dell'Insubria - Como\endgraf
via Valleggio 11\endgraf
I-22100 Como, ITALY}
\email{stefano.pigola{@}uninsubria.it}

\address{Dipartimento di Scienza e Alta Tecnologia - Sezione di Matematica\endgraf
Universit\`a dell'Insubria - Como\endgraf
via Valleggio 11\endgraf
I-22100 Como, ITALY}
\email{alberto.setti{@}uninsubria.it}

\address{Section de Math\'ematiques\endgraf
 \'Ecole Polytechnique F\'ederale de Lausanne
 \endgraf station 8
 \endgraf 1015 Lausanne, Switzerland}
\email{marc.troyanov{@}epfl.ch}


\maketitle

\begin{abstract}
The purpose of this paper is to give a self-contained proof that a complete manifold with more than
one end never supports an $L^{q,p}$-Sobolev inequality ($2 \leq p$, $q\leq
p^{*}$), provided the negative part of its Ricci tensor is small (in a
suitable spectral sense). In the route, we discuss potential theoretic
properties of the ends of a manifold enjoying an $L^{q,p}$-Sobolev inequality. 
\end{abstract}

\section*{Introduction}

{A classic subject in Riemannian geometry is the study of the interplay
between curvature bounds } and the topology of the underlying space. If, on
the one hand, this interplay can take the form of a control of certain global
topological invariants such as the homotopy or the homology groups of the
manifold, on the other hand it can be visible in a control of the complexity
at infinity of the space, e.g., the number of its ends. Recall that an end of
a complete Riemannian manifold {$\left(  M,g\right)  $} with respect to a
selected compact set $K$ is any of the unbounded connected components of
$M\backslash K$. Clearly, by enlarging $K$, the corresponding number of ends
$n(K)$ increases and if $n(K)$ is constantly equal to $1$ we say that $M$ is
connected at infinity. It is a well known consequence of the splitting theorem
by J. Cheeger and D. Gromoll, \cite{CG-Ricci}, that a complete manifold with
non-negative Ricci curvature has at most two ends. Furthermore, if the Ricci
curvature is positive at some point, then we have connectedness at infinity.
According to works by H.-D. Cao, Y. Shen and S. Zhu, \cite{CSZ-MRL}, and P. Li
and J. Wang, \cite{LW-MRL} (see also \cite{PRS-Progress}), the Ricci curvature
assumption in Cheeger-Gromoll conclusion can be considerably relaxed provided
a Sobolev inequality of the form
\begin{equation}
S_{2^{\ast},2}\cdot\left\Vert \varphi\right\Vert _{L^{2^{\ast}}}\leq\left\Vert
\nabla\varphi\right\Vert _{L^{2}},\label{l2-sobolev}%
\end{equation}
with $2^{\ast}=2m/\left(  m-2\right)  $, for some constant $S_{2^{\ast},2}>0$
\ and for every $\varphi\in C_{c}^{\infty}\left(  M\right)  $ holds. {More
precisely, connectedness at infinity holds provided the Ricci curvature
satisfies $Ric\geq-q(x)$ where $q$ is a non negative continuous function on
$M$ such that
\begin{equation}
\int_{M}q(x)\varphi(x)^{2}dvol_{g}\leq\int_{M}|\nabla\varphi|^{2}%
dvol_{g}\label{cond.smallq}%
\end{equation}
for all smooth functions $\varphi$ with compact support in $M$. Note that this
condition means that the function $q\geq0$ is small in the following spectral
sense:
\[
\lambda_{1}^{-\Delta-q\left(  x\right)  }\left(  M\right)  =\inf_{\varphi\in
C_{c}^{1}\left(  M\right)  \backslash\left\{  0\right\}  }\frac{\int
_{M}\left(  \left\vert \nabla\varphi\right\vert ^{2}-q\cdot\varphi^{2}\right)
}{\int_{M}\varphi^{2}}\geq0.
\]
} Observe that, by reversing the viewpoint, a complete manifold disconnected
at infinity and with Ricci tensor subjected to the same conditions cannot
support the $L^{2^{\ast},2}$-Sobolev inequality (\ref{l2-sobolev}). In
general, it is an interesting and difficult problem to understand whether or
not a complete manifold enjoys some $L^{q,p}$-Sobolev inequality
\begin{equation}
S_{q,p}\cdot\left\Vert \varphi\right\Vert _{L^{q}}\leq\left\Vert \nabla
\varphi\right\Vert _{L^{p}},\label{lqp-sobolev}%
\end{equation}
for some constant $\ S_{q,p}>0$ \ and for every $\varphi\in C_{c}^{\infty
}\left(  M\right)  $.


In this respect, we point out that the validity of the $L^{q,p}$-Sobolev
inequality (\ref{lqp-sobolev}), when combined with a Ricci curvature
assumption, implies some constraints on the fundamental group of the
complete Riemannian manifold $M$. Indeed, a complete $m$-dimensional manifold with
non-negative Ricci curvature enjoys (\ref{lqp-sobolev}) for some (hence every)
$1\leq p<m$ and $q=mp/\left(  m-p\right)  $ if and only if the volume growth
is exactly Euclidean, \cite{Va-JFA, CS-Revista}. If this happens, then by a
result due independently to M. Anderson, \cite{An-Topology}, and P. Li,
\cite{Li-Annals}, the fundamental group of the manifold is necessarily finite.

Further interesting connections between topology and $L^{q,p}$-Sobolev
inequalities arise from a seminal work by Pansu, \cite{P-Torino}, recently
extended in \cite{GT-JGA}. Accordingly, the validity of (\ref{lqp-sobolev}) is
related to a \textquotedblleft global\textquotedblright\ cohomology theory
which is sensitive only on the geometry at infinity of the underlying
manifold, the so called $L^{q,p}$-cohomology, and gives information on the
solvability of non-linear differential equations involving the $p$-Laplace
operator (on differential forms). Very quickly, given $1<p,q<+\infty$,
the $L^{q,p}$-cohomology spaces of the complete Riemannian manifold $M$ are defined as follows;
we refer the reader to \cite{GT-JGA} for a detailed exposition. Let $L^{q}\left(
M,\Lambda^{k}\right)  $ denote the Banach space of $L^{q}$-integrable
$k$-forms endowed with the obvious norm $\left\Vert \omega\right\Vert
_{q}=\left(  \int_{M}\left\vert \omega\right\vert ^{q}\right)  ^{1/q}$. The
usual exterior differential $d$ on smooth, compactly supported $k$-forms
extends weakly to $L^{q}\left(  M,\Lambda^{k}\right)  $ and gives rise to the
Banach space%
\[
\Omega^k_{q,p}(M)=L^{q}(M,\Lambda^{k})\cap d^{-1}L^{p}((M,\Lambda^{k+1}))
\]
with norm $\left\Vert \omega\right\Vert _{q,p}=\left\Vert \omega\right\Vert
_{q}+\left\Vert d\omega\right\Vert _{p}$. In this way, the weak exterior
differential can be considered as a bounded linear operator $d_{q,p}%
^{k}:\Omega^k_{q,p}(M,\Lambda^{k})\rightarrow L^{p}(M,\Lambda^{k+1})$. Since it
satisfies the usual co-boundary rule $d\circ d=0$ then, as in the classical de
Rham cohomology, one is led to consider the corresponding subspaces of co-cycles
$Z_{q,p}^{k}(M)=\ker d_{q,p}^{k}$ and
co-boundaries $B_{q,p}^{k}(M)=d_{p,q}^{k-1}(\Omega^{k-1}_{q,p}(M))\subset
Z_{q,p}^{k}(M)$, and to define the $k^{\text{th}}$ space of the $L^{q,p}
$-cohomology of $M$ by setting
\[
H_{q,p}^{k}(M)=Z_{q,p}^{k}(M)\diagup B_{q,p}^{k}(M).
\]
By continuity, $Z_{q,p}^{k}(M)$ is always closed, hence Banach, whereas $B_{q,p}^{k}(M)$ could be not.
In case $B_{q,p}^{k}(M)=\overline{B_{q,p}^{k}(M)}$, the $L^{q,p}%
$-cohomology space $H_{q,p}^{k}(M)$ is said to be reduced. If $M$ is compact and if $1<p,q<\infty$ and  $1/p-1/q \leq 1/m$
then all the cohomology spaces $H_{q,p}^{k}(M)$ are
reduced and coincide with the usual de Rham spaces $H_{dR}^{k}(M)$ \cite[Theorem 12.10]{GT-JGA}.

On the other hand, if $M$ is bi-Lipschitz equivalent to $M^{\prime}$ \ then $H_{q,p}^{k}(M)\simeq
H_{q,p}^{k}(M^{\prime})$. In particular, if $M$ is not compact, its $L^{q,p}$-cohomology is not
affected by perturbing the Riemannian metric on a compact set. Now, in the
language of $L^{q,p}$-cohomology, the validity of the Sobolev inequality
(\ref{lqp-sobolev}) means precisely that the first $L^{q,p}$-cohomology space
of $M$ is reduced. Whence, if $M$ is simply connected and $H_{p,r}^{1}%
(M)\neq0$ for some $r>1$, using a variational argument one can show that $M$
supports a non-constant $p$-harmonic function with finite $p$-energy.
Reversing the viewpoint, this circle of ideas shows that, under the validity
of (\ref{lqp-sobolev}), if the first $L^{q,p}$-cohomology group vanishes, then
the existence of a non-constant $p$-harmonic function $v:M\rightarrow
\mathbb{R}$ with finite $p$-energy $\left\vert \nabla v\right\vert \in
L^{p}(M)$ implies that $M$ is not simply connected.\medskip

This brief and quite informal overview should have given some idea of the influence of $L^{q,p}$-Sobolev inequalities
on the topology and the complexity at infinity of the space.


The goal of this paper is to give a complete and self contained proof
of the following Theorem that extends to every $p\geq 2$ the results by Cao-Shen-Zhu and Li-Wang alluded to
at the beginning of paper.

\begin{theorem}\label{th_main}
Let $(M,g)$ be a complete non compact Riemannian manifold of dimension $m$. Let
$q>p\geq2$ be such that%
\[
  \frac{1}{p}-\frac{1}{q}\leq\frac{1}{m},
\]
and assume that $M$ supports an $L^{q,p}$-Sobolev inequality of the type
\[
S_{q,p}\left\Vert \varphi\right\Vert _{L^{q}}\leq\left\Vert \nabla
\varphi\right\Vert _{L^{p}},
\]
for some constant $\ S_{q,p}>0$ \ and for every $\varphi\in C_{c}^{\infty
}(M)  $. Assume that the Ricci tensor of $M$ is such that
\begin{equation}\label{ineq.ricci}
^{M}\!\Ric\geq-q\left(  x\right)  \text{ on }M
\end{equation}
for a suitable function $q\in C^0(M)  $. If the Schrodinger
operator $L=\Delta+Hq\left(  x\right)  $ \ satisfies%
\begin{equation}\label{lamba1}%
 \lambda_{1}^{-L}(M)  \geq0,
\end{equation}
for some constant $H>p^{2}/4\left(  p-1\right)  $, then, $M$ is connected at infinity,
i.e. for any compact set $F \subset M$, the complement $M\setminus F$ has exactly one
unbounded connected component.
\end{theorem}

The paper is organized as follows. In section 1 we give a rapid proof of the main Theorem, which is based on three facts
from the non linear potential theory of Riemannian manifolds. In section 2 we provide the necessary
background and in sections 3--5, we give detailed proofs
of the potential theoretical results that are used.
Our proofs and viewpoints are  independent of the existing literature (although we provide all the relevant references)
and somewhat more direct.
For instance, the potential theoretic properties of the ends are studied via a direct use of the doubling procedure and the
equivalence between $p$-parabolicity in terms of $p$-capacity and $p$-subharmonic functions is obtained without the use of
the non-linear Green kernel introduced by I. Holopainen.
In the route, we also deduce a form of the Ahlfors maximum principle characterization of $p$-parabolicity
using exterior domains and we extend a result by G. Carron concerning Sobolev inequalities outside a compact set.

\section{Proof of the main theorem}

We argue by contradiction. Assume that  the complement $M\setminus F$ of a given compact subset $F\subset M$, contains at least two disjoint unbounded connected component $E_1,E_2$. From a theorem by S. Buckley and P. Koskela (see \cite{BuKo-MathZ} and Theorem \ref{th_hypends} below),  we know that $E_1,E_2$ are $p$-\emph{hyperbolic}. This means that for any compact subset $K_i \subset \overline{E}_i$ there exists $\alpha_i>0$ such that
\begin{equation}\label{cond.hypE}
 \int_M |\nabla u|^pd\vol_g \geq \alpha_i
\end{equation}
for any $u\in C^{1}_c(M)$ such that $u \geq 1$ on $K_i$.

Using the $p$-hyperbolicity of $E_1,E_2$ together with a result of I. Holopainen  (see  \cite{Ho-Revista} and  Theorem \ref{th_pharm} below),  there exists a non-constant $p$-harmonic function $w$ of finite $p$-energy, that is a function
$w\in C^{1}(M)$ such that
$$
 \Div\left(|\nabla w|^{p-2} \nabla w \right)= 0 \qquad \text{and} \qquad
  \int_M |\nabla w|^pd\vol_g <\infty.
$$
The conclusion follows now from a Liouville type theorem recently proved by G. Veronelli and the first author (see   \cite{PV-Geom},  and
Theorem \ref{thm_liouville}  below). This theorem says that under the conditions (\ref{ineq.ricci})  (\ref{lamba1}) on the Ricci curvature, every $p$-harmonic function $u\in C^{1}(M)$ with finite $p$-energy
$\left\vert \nabla u\right\vert \in L^{p}(M) $ must be constant if  $p\geq2$.
Applying this result to the function $w$ above gives us the required contradiction.

\qed

\bigskip

To sum up, the main theorem follows from (1) the fact that the Sobolev inequality implies that $(M,g)$ has only $p$-hyperbolic ends,
(2) the fact that a manifold with more than one hyperbolic end carries a non constant $p$-harmonic function with finite $p$-energy
and (3)  a Liouville type theorem saying that under our curvature assumption, every $p$-harmonic function with finite $p$-energy
on $M$ is constant. In the following sections, we give precise statements and independent complete proofs for these three facts.

\section{Preliminary results from non-linear potential theory}

A basic notion in geometric analysis is that of $p$-parabolicity and $p$-hyperbolicity
of a Riemannian manifold ($1\leq p < \infty$),   see e.g.  \cite{Ho-Dissertation,Tr-Parab}.
Recall first that the $p$-capacity of a compact set $K$ in
a Riemannian manifold $(M,g)$ is defined as
\[
\mathrm{cap}_{p}\left(  K\right)  =\inf\int_{M}\left\vert \nabla
\varphi\right\vert ^{p}  d\vol_g,
\]
where the infimum is taken with respect to all functions $\varphi\in
C_{c}^{1}(M)$ such that $\varphi\geq1$ on $K$.

\begin{definition}
\label{def_par} A Riemannian manifold $(M,g)$ is said to be $p$-\emph{parabolic}
if the $p$-capacity of every compact set $K\subset M$ vanishes.
The manifold is $p$-\emph{hyperbolic} if it contains a compact set of positive
$p$-capacity.
\end{definition}
Compact manifolds are obviously $p$-parabolic for any $p$.
It is not hard to prove that on a connected $p$-hyperbolic manifold, every
compact set of positive measure has positive
$p$-capacity.

\medskip

The next result gives further equivalent
characterizations of $p$-parabolicity, we show in particular that $p$-parabolicity is
equivalent to an exterior maximum principle. Note that the equivalence between (ii)
-- (iv) below is proved following arguments valid in the case $p=2$ (see e.g.
\cite{PRS-Progress}) while the equivalence with condition (v) is a result in
\cite{GT-axiomatic}. Furthermore, the equivalence (i) -- (ii) was already
observed in \cite{Ho-Dissertation} using the non-linear Green function
introduced by the author, and can be deduced from the results in \cite{Tr-PAMS}.
However, we provide a new and direct argument. Finally, to the best of our
knowledge, the explicit equivalence (iii) -- (iv) has never been observed before.

\begin{theorem}
\label{th_parabcrit} Let $\left(  M,g \right)  $ be
a complete Riemannian manifold. The following conditions are equivalent.

\begin{enumerate}
\item[(i)] $M$ is $p$-parabolic.

\item[(ii)] If $u\in C^0(M)\cap W^{1,p}_{loc}(M)$ is a bounded above weak
solution of $\Delta_{p} u\geq0$ then $u$ is constant.

\item[(iii)] There exists a relatively compact domain $D$ in $M$ such that,
for every function $\varphi\in C(M\setminus D)\cap W^{1,p}_{loc}
(M\setminus\overline D)$ which is bounded above and satisfies $\Delta_{p}
\varphi\geq0$ weakly in $M\setminus\overline D$, $\sup_{M\setminus D}%
\varphi=\max_{\partial D}\varphi$.

\item[(iv)] For every open set $\Omega\subset M$ with $\partial \Omega\ne \emptyset$,
and for every $\psi\in C\left(
\bar{\Omega}\right)  \cap W_{loc}^{1,p}(\Omega)$ which is bounded above and
satisfies $\Delta_{p}\psi\geq0$ weakly on $\Omega$, $\sup_{\Omega}\psi
=\sup_{\partial\Omega}\psi$.

\item[(v)] There exists a compact set $K\subset M$ with the following
property. For every constant $C>0$, \ there exists a compactly supported
function $v\in W^{1,p}(M)  \cap C^0(M)  $ such that%
\[
\left\Vert v\right\Vert _{L^{p}\left(  K\right)  } \geq C\left\Vert \nabla
v\right\Vert _{L^{p}(M)  }.
\]

\end{enumerate}
\end{theorem}

\begin{proof}
(i) $\Rightarrow$ (ii). Let $M$ be $p$-parabolic, so that, for every compact
set $K$, $\mathrm{cap}_{p}(K)=0$, and assume by contradiction that there
exists a positive, $p$-superharmonic function $u$. By translating and scaling,
we may assume that $\sup u>1$ and $\inf u=0$. Note that, by the strong maximum
principle (see, e.g., \cite{HKM} Theorem 7.12) $u$ is strictly positive on
$M$. Next let $D$ be a relatively compact domain with smooth boundary
contained in the superlevel set $\{u>1\}$ and let $D_{i}$ be an exhaustion of
$M$ consisting of relatively compact domains with smooth boundary such that
$\overline{D}\subset\subset D_{1}$, and for every $i$ let $h_{i}$ be the
solution of the the Dirichlet problem%
\[
\left\{
\begin{array}
[c]{l}%
\Delta_{p}h_{i}=0\text{, on }D_{i}\setminus D\\
h_{i}=1\text{, on }\partial D\\
h_{i}=0\text{, on }\partial D_{i}.
\end{array}
\right.
\]
By a result of Tolksdorf \cite{Tolksdorf}, $h_{i}\in C_{loc}^{1,\alpha}\left(
D_{i}\setminus\overline{D}\right)  $. Furthermore, since $D$ and $D_{i}$ have
smooth boundaries, applying Theorem 6.27 in \cite{HKM} with $\theta$ any
smooth extension of the piecewise function%
\[
\theta_{0}=\left\{
\begin{array}
[c]{l}%
1\text{, on }\partial D\\
0\text{, on }\partial D_{i},
\end{array}
\right.
\]
we deduce that $h_{i}$ is continuous on $\overline{D}_{i}\setminus D$. By the
strong maximum principle, we have $0<h_{i}<1$ in $D_{i}\setminus\overline{D}$ and
using the comparison
principle, \cite[Lemma 3.18]{HKM},   we deduce that
$\left\{  h_{i}\right\}  $ is an increasing sequence. Hence, by the Harnack principle,
$\left\{  h_{i}\right\}  $ converges locally uniformly on $M\setminus D$ a
function $h$ which is continuous on $M\setminus D$, $p$-harmonic on
$M\setminus\overline{D}$ and satisfies $0<h\leq1$ on $M\setminus\overline{D}%
$\ and $h=1$ on $\partial D$. Again, $h\in C\left(  M\setminus D\right)  \cap
C_{loc}^{1,\alpha}\left(  M\setminus\bar{D}\right)  $.

Moreover, since $h_{i}$ is the $p$-equilibrium potential of the condenser
$(\overline{D},D_{i})$,
\[
\mathrm{cap}_{p}\left(  \overline{D},D_{i}\right)  =\int\left\vert \nabla
h_{i}\right\vert ^{p}=\inf\int\left\vert \nabla\varphi\right\vert ^{p},
\]
where the infimum is taken with respect to $\varphi\in C_{c}^{\infty}\left(
D_{i}\right)  $ such that $\varphi=1$ on $\partial D$. Think of each $h_{i}$
extended to be zero off $D_{i}$. Therefore $\left\{  \int_{M\setminus
\overline{D}}\left\vert \nabla h_{i}\right\vert ^{p}\right\}  $ is decreasing
and the sequence $\left\{  h_{i}\right\}  \subset W^{1,p}\left(
\Omega\right)  $ is bounded on every compact domain $\Omega$ of $M\setminus
\overline{D}$. By the weak compactness theorem, see, e.g., Theorem 1.32 in
\cite{HKM}, $h\in W^{1,p}\left(  \Omega\right)  $, and $\nabla h_{i}%
\rightarrow\nabla h$ weakly in $L^{p}\left(  \Omega\right)  $. In particular,
\[
\int_{\Omega}\left\vert \nabla h\right\vert ^{p}\leq\liminf_{i\rightarrow
+\infty}\int_{D_{i}\setminus D}\left\vert \nabla h_{i}\right\vert ^{p}.
\]
On the other hand, it follows easily from the definition of capacity, that
$\lim_{i}\mathrm{cap}_{p}(\overline{D},D_{i})=\mathrm{cap}_{p}(\overline
{D})=0$. Thus, letting $\Omega\nearrow M\setminus\overline{D}$ we conclude
that
\[
\int_{M\setminus\overline{D}}|\nabla h|^{p}=0,
\]
so that $h$ is constant, and since $h=1$ on $\partial D$, $h\equiv1$. Finally,
since $u$ is $p$-superharmonic and $u>h_{i}$ on $\partial D\cup\partial D_{i}%
$, by the comparison principle, $u\geq h_{i}$ on $D_{i}\setminus D$, and
letting $i\rightarrow\infty$ we conclude that $u\geq1$ on $M$, contradiction.

(ii) $\Rightarrow$ (i). Given a relatively compact domain $D$, let $h_{i}$ and
$h$ be the functions constructed above, and extend $h$ to be $1$ in $D$, so
that $h$ is continuous on $M$, bounded, and satisfies $\Delta_{p}h\leq0$
weakly on $M$. Thus (ii) implies that $h$ is identically equal to $1$. On the
other hand, since the functions $h_{i}$ belong to $W_{0}^{1,p}(M)$, Lemma 1.33
in \cite{HKM} shows that $\nabla h_{i}$ converges to $\nabla h$ weakly in
$L^{p}(M)$. By Mazur's Lemma (see Lemma 1.29 in \cite{HKM}) there exists a
sequence $v_{k}$ of convex combinations of the $h_{i}$'s such that $\nabla
v_{k}$ converges to $\nabla h$ strongly in $L^{p}$. Thus $v_{k}$ is
continuous, compactly supported, identically equal to $1$ on $\overline{D}$
(because so are all the $h_{i}$'s) and $\int_{M}|\nabla v_{k}|^{p}%
\rightarrow\int|\nabla h|^{p}=0$, showing that $\mathrm{cap}_{p}(\overline
{D})=0$, and $M$ is $p$-parabolic.

(iii) $\Rightarrow$ (iv). Assume that (iii) holds, and suppose by
contradiction that there exist a domain $\Omega$ and a function $\psi$ as in
(iv) for which $\sup_{\partial\Omega}\psi<\sup_{\Omega}\psi$. Note that, by
the strong maximum principle, $\Omega$ is unbounded. Choose $0<\varepsilon
<\sup_{\Omega}\psi-\sup_{\partial\Omega}\psi$ sufficiently near to
$\sup_{\Omega}\psi-\sup_{\partial\Omega}\psi$ so that $\overline{D}%
\cap\left\{  \psi>\sup_{\partial\Omega}\psi+\varepsilon\right\}  =\emptyset$.
This is possible according to the strong maximum principle, because
$\overline{D}$ is compact. Define $\tilde{\psi}\in C^0(M)  \cap
W_{loc}^{1,p}(M)  $ by setting%
\[
\tilde{\psi}(x)=\max\{\sup_{\partial\Omega}\psi+\varepsilon,\psi(x)\}
\]
and note that $\Delta_{p}\tilde{\psi}\geq0$ on $M$. According to property
(iii),%
\[
\max_{\partial D}\tilde{\psi}=\sup_{M\backslash D}\tilde{\psi}.
\]
However, since $\overline{D}\cap\left\{  \psi>\sup_{\partial\Omega}%
\psi+\varepsilon\right\}  =\emptyset$,%
\[
\max_{\partial D}\tilde{\psi}=\sup_{\partial\Omega}\psi+\varepsilon
<\sup_{\Omega}\psi,
\]
while%
\[
\sup_{M\backslash D}\tilde{\psi}=\sup_{\Omega}\psi.
\]
The contradiction completes the proof.

(iv)$\Rightarrow$(iii). Trivial.

(iv)$\Rightarrow$(ii). Assume by contradiction that there exists $u\in
C^0(M)\cap W_{loc}^{1,p}(M)$ which is non-constant, bounded above and satisfies
$\Delta_{p}u\geq0$ weakly on $M$. Given $\gamma<\sup u$, the set
$\Omega_{\gamma}=\{u>\gamma\}$ is open, and $u$ is continuous and bounded
above in $\overline{\Omega}_{\gamma}$, satisfies $\Delta_{p}u\geq0$ weakly in
$\Omega_{\gamma}$ and $\max_{\partial\Omega_{\gamma}}u<sup_{\Omega_{\gamma}}%
u$, contradicting (iv).

(ii)$\Rightarrow$(iv). If there exists $\psi\in C\left(  \overline{\Omega
}\right)  \cap W_{loc}^{1,p}\left(  \Omega\right)  $ satisfing $\Delta_{p}%
\psi\geq0$ and $\sup_{\Omega}\psi>\max_{\partial\Omega}\psi+2\varepsilon$, for
some $\varepsilon>0$, then%
\[
\psi_{\varepsilon}=\left\{
\begin{array}
[c]{ll}%
\max\left\{  \psi,\max_{\partial\Omega}\psi+\varepsilon\right\}  & \text{in
}\Omega\\
\max_{\partial\Omega}+\varepsilon & \text{in }M\backslash\Omega,
\end{array}
\right.
\]
is a non-constant, bounded above, weak solution of $\Delta_{p}\psi
_{\varepsilon}\geq0$ on $M$. This contradicts (ii).

For the equivalence (i) $\Leftrightarrow$ (v), see \cite{GT-axiomatic} Theorem 3.1.
\end{proof}

We now localize the concept of parabolicity on a given end. Recall that, by
definition, an end $E$ of $M$ with respect to a compact domain $F$ is any of
the unbounded connected components of $M\backslash F$.

\begin{definition}
\label{def_parends}An end $E$ of the Riemannian manifold $\left(
M,g \right)  $ is said to be $p$-parabolic if, for
every compact set $K\subset\bar{E}$,%
\[
\mathrm{cap}_{p}\left(  K,E\right)  =\inf\int_{E}\left\vert \nabla
\varphi\right\vert ^{p}=0,
\]
where the infimum is taken with respect to all $\varphi\in C_{c}^{\infty
}\left(  \bar{E}\right)  $ such that $\varphi\geq1$ on $K$.
\end{definition}

We have the following characterizations of the parabolicity of ends.

\begin{definition}
The Riemannian double of a manifold $E$ with smooth, compact boundary
$\partial E$ is defined to be  a smooth Riemannian manifold (without boundary) $\mathcal{D}%
(E)$ such that (i) $\mathcal{D}(E)$ is complete (ii) $\mathcal{D}(E)$ is
homeomorphic to the topological double of $E$ and (iii) there is a compact set
$K\subset\mathcal{D}(E)$ such that $\mathcal{D}(E)\setminus K$ has two
connected components, both isometric to $E$.
\end{definition}
Observe that this is not uniquely
defined, but all such \textquotedblleft doubles\textquotedblright are
bilipshitz equivalent.

\begin{theorem}
\label{th_parabcritends}An end $E$ with smooth boundary $\partial E$ is
$p$-parabolic if {and only if}
either  {one} of the following equivalent conditions is satisfied:
\begin{enumerate}
\item[(i)] For every continuous $\phi:\bar{E}\rightarrow\mathbb{R}$ which is
bounded above and $p$-subharmonic, $\sup_{E}\phi=\max_{\partial E}\phi$.

\item[(ii)] The (Riemannian) double $\mathcal{D}\left(  E\right)  $ of $E$ is
a $p$-parabolic manifold without boundary.
\end{enumerate}
\end{theorem}

Condition (i) in Theorem~\ref{th_parabcritends} yields easily the following
necessary and sufficient condition for an end to be $p$-hyperbolic.

\begin{corollary}
\label{p-hyperbolicity ends} An end $E$ is $p$-hyperbolic if and only if there
exists a function $\psi\in C(\overline{E})\cap W_{loc}^{1,p}(E)$ which is
$p$-superharmonic and such that $\inf_{E}\psi=0$ and $\psi\geq1$ on $\partial
E.$
\end{corollary}

Corollary~\ref{p-hyperbolicity ends} allows us to obtain the existence of
special $p$-harmonic functions on $p$-hyperbolic ends (whose existence, in
view of Theorem~\ref{th_parabcrit} in fact characterizes $p$-hyperbolic ends).

\begin{lemma}
\label{lemma_pharmends}Let $E$ be a $p$-hyperbolic end of $\left(
M,g \right)  $ with smooth boundary. Then, there
exists a non-constant $\ p$-harmonic function $h\in C\left(  \bar{E}\right)
\cap C_{loc}^{1,\alpha}\left(  E\right)  $ such that:

\begin{enumerate}
\item[(1)] $0<h\leq1$ in $\bar{E}$,

\item[(2)] $h=1$ on $\partial E$,

\item[(3)] $\inf_{\bar{E}}h=0,$

\item[(4)] $\left\vert \nabla h\right\vert \in L^{p}\left(  \bar{E}\right)  .$
\end{enumerate}
\end{lemma}

\begin{proof}
Take a smooth exhaustion $D_{i}$ of $M$ with $\partial E\subset D_{0}$. Set
$E_{i}=E\cap D_{i}$ and solve the Dirichlet problem%
\[
\left\{
\begin{array}
[c]{ll}%
\Delta_{p}h_{i}=0\text{,} & \text{ on }E_{i}\\
h_{i}=1\text{,} & \text{ on }\partial E\\
h_{i}=0\text{,} & \text{ on }\partial D_{i}\cap E.
\end{array}
\right.
\]
By the arguments used in the proof of Theorem~\ref{th_parabcrit}, $h_{i}\in
C_{loc}^{1,\alpha}\left(  E_{i}\right)  \cap C(\overline{E}_{i})$, $0<h_{i}<1$
in $E_{i}$, it is increasing and converges (locally uniformly) to a
$p$-harmonic function $h$ on $h\in C\left(  \bar{E}\right)  \cap
C_{loc}^{1,\alpha}\left(  E\right)  $\ satisfying $0<h\leq1$ \ and $h=1$ on
$\partial E$. Since $E$ is $p$-hyperbolic, there exists a function $\psi$ with
the properties listed in Corollary~\ref{p-hyperbolicity ends}. By the
comparison principle, $h_{i}\leq\psi$ for every $i$, and passing to the limit,
$h\leq\psi$, so that $\inf_{E}h=0$ and in particular $h$ is non-constant.

To prove that $h$ has finite $p$-energy we argue as in the proof of
Theorem~\ref{th_parabcrit}, to show that $\left\{  \int_{E}\left\vert \nabla
h_{i}\right\vert ^{p}\right\}  $ (where $h_{i}$ is extended to $E$ by setting
it equal to $0$ in $E\setminus E_{i}$) is decreasing and, by Lemma 1.33 in
\cite{HKM}, $\nabla h\in L^{p}(E)$ and $\nabla h_{i}$ converges to $\nabla h$
weakly in $L^{p}(E)$.
\end{proof}

\begin{remark}
\label{rm3}Suppose that the end $E$ is $p$-parabolic. Then, the same
construction works but, in this case, by the boundary maximum principle
characterization of parabolicity, we have $h\equiv1$.
\end{remark}

\section{Sobolev inequalities, volume and hyperbolicity of ends}

In this section we show that the validity of an $L^{q,p}$-Sobolev inequality
implies that each of the ends of the underlying manifold is $p$-hyperbolic.
Unlike previous investigations by Li-Wang \cite{LW-MRL} and Cao-Shen-Zhu \cite{CSZ-MRL} for $p=2$, and Buckley-Koskela \cite{BuKo-MathZ}
for general $p$ and general metric ambient spaces, our strategy is to use in a natural
way the doubling construction
on the given end, thus reducing the study to the case of a manifold without boundary.
Technical difficulties arising from the validity of the Sobolev inequality only outside
a compact set are overcome by extending a previous result by Carron.

We begin by describing the effect on volume growth of the validity of a Sobolev inequality.
It is elementary to show that if an $L^{q,p}$-Sobolev inequality holds on a
manifold then the manifold has infinite volume. Indeed, having fixed $x_{o}$
in $M$ we consider a family $\left\{  \varphi_{R}\right\}  _{R>0}$ of cut-off
functions satisfying: (a) $0\leq\varphi_{R}\leq1$; (b) $\varphi_{R}=1$ on
$B_{R/2}\left(  x_{o}\right)  $; (c) \textrm{supp}$\left(  \varphi_{R}\right)
\subset B_{R}(x_{o})$; (d) $\left\vert \nabla\varphi_{R}\right\vert \leq4/R$
on $M$. Using $\varphi_{R}$ into the Sobolev inequality gives
\[
S_{q,p}\mathrm{vol}\left(  B_{R/2}\left(  x_{o}\right)  \right)  ^{1/q}\leq
S_{q,p}\left\Vert \varphi_{R}\right\Vert _{L^{q}}\leq\left\Vert \nabla
\varphi_{R}\right\Vert _{L^{p}}\leq\frac{4}{R}\mathrm{vol}\left(  B_{R}\left(
o\right)  \right)  ^{1/p},
\]
which, in turn, implies the non-uniform estimate%
\[
\mathrm{vol}\left(  B_{R}\left(  o\right)  \right)  \geq CR^{p},
\]
for every $R\geq1$ and for some constant $C=\left(  4^{-1}S_{q,p}%
\mathrm{vol}\left(  B_{1}\left(  o\right)  \right)  ^{1/q}\right)  ^{p}>0$. In
particular $\mathrm{vol}(M)  =+\infty$ and at least one of the
ends of $M$ has infinite volume.

In order to extend this conclusion to each individual end
we can use a uniform volume estimate whose principle can be traced back to papers by
G. Carron, \cite{C-Spectral} and  K. Akutagawa \cite{akutagawa}, see also \cite[lemma 2.2 ]{H-NA} and
\cite[theorem 3.1.5]{scoste}. We state this estimate in a form suitable for our purposes.

\begin{proposition}
\label{th_volume} Let $E$ be an  end of the complete manifold $M$ with
respect to the compact set $F$ and assume that the $L^{q,p}$-Sobolev
inequality (\ref{lqp-sobolev}) holds on $E$, for some $q>p\geq1$. Then there
exists positive constant $C_{1}$  depending only on $p,q$ and
$S_{q,p}$ such that, for every geodesic ball $B_{R}(x_{0})\subset E$
\begin{equation}
vol\left(  B_{R}\left(  x_{0}\right)  \right)  \geq C_{1}R^{\frac{pq}{q-p}}.
\label{volume}%
\end{equation}
In particular, if $F\subset B_{R_{0}}(o)$, then for every $x_{0}\in E$ with
$d(x_{0},o)\geq R+R_{0}$ the ball $B_{R}(x_{0})$ is contained in $E,$ and $E$ has
infinite volume.
\end{proposition}

\begin{proof}
For every $\Omega\subset E$ let
\[
\lambda\left(  \Omega\right)  =\inf\frac{\int_{\Omega}\left\vert \nabla
\varphi\right\vert ^{p}}{\int_{\Omega}\left\vert \varphi\right\vert ^{p}},
\]
the infimum being taken with respect to all $\varphi\in W_{c}^{1,p}\left(
\Omega\right)  $, $\varphi\not \equiv 0$. By the Sobolev and H\"older
inequalities, for every such $\varphi$ we have
\[
\int_{\Omega}\varphi^{p} \leq\mathrm{vol} (\Omega)^{\frac{q-p}{q}} \left(
\int_{\Omega}\varphi^{q}\right)  ^{\frac pq} \leq\left(  S_{q,p}
||\nabla\varphi||_{p}\right)  ^{p},
\]
and therefore
\begin{equation}
vol\left(  \Omega\right)  ^{\frac{q-p}{q}}\lambda\left(  \Omega\right)  \geq
S_{q,p}^{p}. \label{lamba below}%
\end{equation}
On the other hand, choosing $\Omega=B_{R}\left(  x_{0}\right)  $ and%
\[
\varphi\left(  x\right)  =R-d\left(  x,x_{0}\right)
\]
we deduce that
\begin{align}
\lambda\left(  B_{R}\left(  x_{0}\right)  \right)   &  \leq\frac{vol\left(
B_{R}\left(  x_{0}\right)  \right)  }{\int_{B_{R}\left(  x_{0}\right)
}\left(  R-d\left(  x,x_{0}\right)  \right)  ^{p}}\label{lamba above}\\
&  \leq\frac{vol\left(  B_{R}\left(  x_{0}\right)  \right)  }{\int
_{B_{R/2}\left(  x_{0}\right)  }\left(  R-d\left(  x,x_{0}\right)  \right)
^{p}}\nonumber\\
&  \leq\frac{2^{p}vol\left(  B_{R}\left(  x_{0}\right)  \right)  }%
{R^{p}vol\left(  B_{R/2}\left(  x_{0}\right)  \right)  }.\nonumber
\end{align}
Combining (\ref{lamba below}) and (\ref{lamba above}) we obtain
\[
vol\left(  B_{R}\left(  x_{0}\right)  \right)  ^{1+\frac{q-p}{q}}\geq
2^{-p}S_{q,p}^{p}R^{p}vol\left(  B_{R/2}\left(  x_{0}\right)  \right)  ,
\]
i.e.,%
\[
vol\left(  B_{R}\left(  x_{0}\right)  \right)  \geq\left(  2^{-p}S_{q,p}%
^{p}R^{p}\right)  ^{\alpha}vol\left(  B_{R/2}\left(  x_{0}\right)  \right)
^{\alpha},
\]
with%
\[
0<\alpha=\frac{1}{1+\frac{q-p}{q}}<1.
\]
Iterating this inequality $k$-times yields%
\[
vol\left(  B_{R}\left(  x_{0}\right)  \right)  \geq2^{-p\alpha\sum_{j=1}%
^{k}j\alpha^{j}}\left(  2^{-p}S_{q,p}^{p}R^{p}\right)  ^{\sum_{j=1}^{k}%
\alpha^{j}}vol\left(  B_{R/2^{k}}\left(  x_{0}\right)  \right)  ^{\alpha^{k}%
}.
\]
Since%
\[
volB_{r}\left(  x_{0}\right)  \sim\omega_{0}r^{m}\text{ as } r\rightarrow0
\,\,\,(m=\mathrm{dim}\,M),
\]
for $k$ large enough
\[
vol\left(  B_{R/2^{k}}\left(  x_{0}\right)  \right)  ^{\alpha^{k}} \geq\left(
\frac12 \omega_{0} R ^{m } 2^{-km }\right)  ^{\alpha^{k}}
\]
and letting $k\rightarrow+\infty$ finally gives%
\[
vol\left(  B_{R}\left(  x_{0}\right)  \right)  \geq2^{-p\bar{\alpha}}\left(
2^{-p}S_{q,p}^{p}R^{p}\right)  ^{\frac{\alpha}{1-\alpha}},
\]
where%
\[
\bar{\alpha}=\sum_{j=1}^{+\infty}j\alpha^{j},
\]
and estimate (\ref{volume}) holds since $\frac{p \alpha}{1-\alpha} = \frac{pq}{q-p}$.

To prove the second statement, assume that $x_{0}\in E$ is such that
$d(x_{0},o)\geq R+R_{0}$, and consider the geodesic ball $B_{R}(x_{0})$. If $x
\in\overline{B_{R_{0}}(o)}$, then by the triangle inequality,
\[
d\left(  x_{0},x\right)  \geq d\left(  x_{0},o\right)  -d\left(  o,x\right)
\geq R,
\]
proving that $B_{R}\left(  \bar x\right)  \cap\overline{B_{R_{0}}\left(
o\right)  } =\emptyset$. On the other hand, if $E^{\prime}$ is a second
connected component of $M\backslash K$ and $x^{\prime\prime}\setminus
B_{R_{0}}(o)$, let $\sigma$ be a minimizing geodesic from $x_{0}$ \ to
$x^{\prime}$. By continuity, $\sigma$ must intersect $\partial B_{R_{0}%
}\left(  o\right)  $ at some point $x_{1}$ and
\[
d\left(  x_{0},x^{\prime}\right)  =\ell\left(  \sigma\right)  =d(x^{\prime
},x_{1})+ d(x_{1}, x_{0}) > d(x_{1},\bar x)\geq R.
\]
Therefore $B_{R}\left(  x_{0}\right)  \cap E^{\prime}=\emptyset$ and we
conclude that $B_{R}(x_{0})\subset E$. Since $x_{0}\in E$ can be chosen in
such a way that $d\left(  x_{0},o\right)  $ is arbitrarily large, letting
$E\ni x_{0}\rightarrow\infty$ gives that $vol\left(  E\right)  =+\infty$.
\end{proof}

We next prove that if an $L^{q,p}$-Sobolev inequality holds in the complement of
a compact set of a complete Riemannian manifold, then each end is $p$-hyperbolic.
The result is known for $p=2$ \cite{CSZ-MRL,LW-MRL,PRS-Progress}. The proof
we give here is new and is  based on the observation that if the
$L^{q,p}$ Sobolev inequality (\ref{lqp-sobolev}) holds on $M$ then $M$ is
necessarily $p$-hyperbolic. Indeed, if $\Omega$ is any compact domain then,
for every $\varphi\in C_{c}^{\infty}(M)  $ satisfying
$\varphi\geq1$ on $\Omega$ it holds%
\[
S_{q,p}\mathrm{vol}\left(  \Omega\right)  ^{1/q}\leq S_{q,p}\left\Vert
\varphi\right\Vert _{L^{q}}\leq\left\Vert \nabla\varphi\right\Vert _{L^{p}},
\]
proving that%
\[
\mathrm{cap}_{p}\left(  \Omega\right)  \geq S_{q,p}^{p}\mathrm{vol}\left(
\Omega\right)  ^{p/q}>0.
\]
This shows that $M$ is $p$-hyperbolic, and therefore at least one of its ends
is $p$-hyperbolic. To extend the conclusion to each end $E$ of $M$, we are
naturally led to applying the reasonings to the double $\mathcal{D}\left(
E\right)  $. By the very definition of the double of a manifold, it turns out
that $\mathcal{D}\left(  E\right)  $ supports the Sobolev inequality
(\ref{lqp-sobolev}) outside a compact neighborhood of the glued boundaries.
Accordingly, to conclude that $E$ is $p$-hyperbolic we can make a direct use
of the following very general theorem that extends to any $L^{q,p}$-Sobolev
inequality a previous result by Carron, \cite{C-Duke}.

\begin{theorem}
\label{th_hyp&sobolev}Let $\left(  M,g \right)  $ be
a possibly incomplete Riemannian manifold. Assume that $M$ has infinite volume
and that $M\setminus F$ supports the $L^{q,p}$-Sobolev inequality (\ref{lqp-sobolev}) for some
compact $F\subset M$.  Then, $M$ is $p$-hyperbolic and the same Sobolev
inequality, possibly with a different constant, holds on all of $M$.
\end{theorem}

\begin{remark}
Clearly, if $M$ is complete, according to Proposition \ref{th_volume} the assumption that
$M$ has infinite volume is automatically satisfied.

\end{remark}

\begin{proof}
Let $\Omega$ be a precompact domain with smooth boundary such that
$K\subset\subset\Omega$. Let also $W_{\varepsilon}\approx\partial\Omega
\times\left(  -\varepsilon,\varepsilon\right)  $ be a bicollar neighborhood of
$\partial\Omega$ such that $W_{\varepsilon}\subset M\backslash F$, and let
$\Omega_{\varepsilon}=\Omega\cup W_{\varepsilon}$ and $M_{\varepsilon
}=M\backslash\Omega_{\varepsilon}$. Note that, by assumption, the $L^{q,p}%
$-Sobolev inequality with Sobolev constant $S>0$ holds on $M_{\varepsilon}$.
Furthermore, the same $L^{q,p}$-Sobolev inequality, with some constant
$S_{\varepsilon}>0$ holds on the compact manifold with boundary $\overline
{\Omega_{\varepsilon}}$ (start with the Euclidean $L^{1}$ Sobolev inequality
and use H\"older's inequality a number of times). Now, let $\rho\in C_{c}^{\infty
}(M)  $ be a cut-off function satisfying $0\leq\rho\leq1$,
$\rho=1$ on $\Omega_{\varepsilon/2}$ and $\rho=0$ on $M_{\varepsilon}.$ Next,
for any $v\in C_{c}^{\infty}(M)  $, write $v=\rho v+\left(
1-\rho\right)  v$, and note that $\rho v\in C_{c}^{\infty}\left(
\Omega_{\varepsilon}\right)  $ whereas $\left(  1-\rho\right)  v\in
C_{c}^{\infty}\left(  M_{\varepsilon/2}\right)  $ \ Therefore, we can apply
the respective Sobolev inequalities and get%
\begin{align*}
\left\Vert v\right\Vert _{L^{q}(M)  }  &  \leq\left\Vert
v\rho\right\Vert _{L^{q}\left(  \Omega_{\varepsilon}\right)  }+\left\Vert
v\left(  1-\rho\right)  \right\Vert _{L^{q}\left(  M_{\varepsilon/2}\right)
}\\
&  \leq S_{\varepsilon}^{-1}\left\Vert \nabla\left(  v\rho\right)  \right\Vert
_{L^{p}\left(  \Omega_{\varepsilon}\right)  }+S^{-1}\left\Vert \nabla\left(
v\left(  1-\rho\right)  \right)  \right\Vert _{L^{p}\left(  M_{\varepsilon
/2}\right)  }\\
&  \leq\left(  S_{\varepsilon}^{-1}+S^{-1}\right)  \left\Vert \nabla
v\right\Vert _{L^{p}(M)  }+S_{\varepsilon}^{-1}\left\Vert
v\nabla\rho\right\Vert _{L^{p}\left(  \Omega_{\varepsilon}\backslash
\Omega_{\varepsilon/2}\right)  }+S^{-1}\left\Vert v\nabla\rho\right\Vert
_{L^{p}\left(  \Omega_{\varepsilon}\right)  }\\
&  \leq\left(  S_{\varepsilon}^{-1}+S^{-1}\right)  \left\{  \left\Vert \nabla
v\right\Vert _{L^{p}(M)  }+C\left\Vert v\right\Vert
_{L^{p}\left(  \Omega_{\varepsilon}\right)  }\right\}  ,
\end{align*}
where $C=\max_{M}\left\vert \nabla\rho\right\vert $. \ Summarizing, we have
shown that, for every $v\in C_{c}^{\infty}(M)  ,$%
\begin{equation}
\left\Vert v\right\Vert _{L^{q}(M)  }\leq C_{1}\left\{
\left\Vert \nabla v\right\Vert _{L^{p}(M)  }+\left\Vert
v\right\Vert _{L^{p}\left(  \Omega_{\varepsilon}\right)  }\right\}  ,
\label{hyp&sob1}%
\end{equation}
for a suitable constant $C_{1}>0$.

With this preparation, we now prove that $M$ is $p$-hyperbolic. To this end,
using the fact that $\mathrm{vol}(M)  =+\infty$, we choose a
compact set $\Omega^{\prime}\supset\Omega_{\varepsilon}$ satisfying%
\[
\mathrm{vol}\left(  \Omega^{\prime}\right)^{1/p}  \geq\left(  2C_{1}\right)
^{q}\mathrm{vol}\left(  \Omega_{\varepsilon}\right)  ^{q/p}.
\]
Thus, applying (\ref{hyp&sob1}) with a test function $v\in C_{c}^{\infty
}(M)  $ satisfying $v=1$ on $\Omega^{\prime}$, we deduce%
\[
\mathrm{vol}\left(  \Omega_{\varepsilon}\right)  \leq C_{1}^{-1}%
\mathrm{vol}\left(  \Omega^{\prime}\right)  ^{1/q}-\mathrm{vol}\left(
\Omega_{\varepsilon}\right)  ^{1/p}\leq\left\Vert \nabla v\right\Vert
_{L^{p}(M)  }.
\]
It follows that%
\[
\mathrm{cap}_{p}\left(  \Omega^{\prime}\right)  \geq\mathrm{vol}\left(
\Omega_{\varepsilon}\right)  >0,
\]
proving that $M$ is $p$-hyperbolic.

Finally, we show that the Sobolev inequalities on $\Omega_{\varepsilon}$ and
on $M_{\varepsilon}$ glue together. According to (\ref{hyp&sob1}) it suffices
to prove that there exists a suitable constant $E=E\left(  \Omega
_{\varepsilon}\right)  >0$ such that%
\[
\left\Vert v\right\Vert _{L^{p}\left(  \Omega_{\varepsilon}\right)  }\leq
E\left\Vert \nabla v\right\Vert _{L^{p}(M)  },
\]
for every $v\in C_{c}^{\infty}(M)  $. Since $M$ is $p$%
-hyperbolic, this latter inequality follows from Theorem \ref{th_parabcrit}
(ii).\bigskip
\end{proof}

We can now prove the result announced at the beginning of this section.

\begin{theorem}
\label{th_hypends} Every end of a complete Riemannian manifold $\left(
M,g \right)  $ supporting the $L^{q,p}$-Sobolev
inequality (\ref{lqp-sobolev}) for some $q>p\geq1$ is $p$-hyperbolic and, in particular, has
infinite volume.
\end{theorem}
\begin{proof}
Let $E$ be an end with smooth boundary of the complete manifold $M$
supporting the $L^{q,p}$-Sobolev inequality (\ref{lqp-sobolev}). We shall
prove that the double $\mathcal{D}\left(  E\right)  $ of $E$ is a
$p$-hyperbolic manifold (without boundary). To this purpose, we note that
$\mathcal{D}\left(  E\right)  $ has infinite volume because, by the first part
of Theorem \ref{th_hypends}, $E$ itself has infinite volume. Furthermore, $E$
enjoys the Sobolev inequality (\ref{lqp-sobolev}) outside a compact
neighborhood of the glued boundaries. Therefore, a direct application of
Theorem \ref{th_hyp&sobolev} yields that $\mathcal{D}\left(  E\right)  $ is a
$p$-hyperbolic manifold, as desired.
\end{proof}

\begin{corollary}
Suppose that the complete manifold $M$ has (at least) one $p$-parabolic end.
Then the $L^{q,p}$-Sobolev inequality (\ref{lqp-sobolev}) fails.
\end{corollary}

\section{$p$-harmonic functions with finite $p$-energy}

This section aims to giving  a simple independent proof of a result of  I. Holopainen, \cite{Ho-Revista},
which extends to the nonlinear setting previous results of the Li-Tam theory, \cite{LT-JDG}.
Related results may be found in the paper by S.W. Kim, and Y.H. Lee, \cite{Kim-Kor}.

\begin{theorem}
\label{th_pharm}Let $\left(  M,g \right)  $ be a
Riemannian manifold with at least two $\ p$-hyperbolic ends (with respect to
some smooth, compact domain).Then, there exists a non-constant , bounded
$p$-harmonic function $u\in C^0(M)  \cap C_{loc}^{1,\alpha}\left(
M\right)  $ satisfying $\left\vert \nabla u\right\vert \in L^{p}\left(
M\right)  $.
\end{theorem}

\begin{proof}
Let $E_{i}$ be the ends of $M$ with respect to the smooth domain
$\Omega\subset\subset M$. By assumption, we may suppose that  $E_{1}$ and $E_{2}$ are
$p$-hyperbolic. Let $\left\{  D_{t}\right\}
_{t\in\mathbb{N}}$ \ be a smooth exhaustion of $M$ and set $E_{j,t}=E_{j}\cap
D_{t}$.

For every $t\in\mathbb{N}$, let $u_{t}\in C_{loc}^{1,\alpha}\left(
D_{t}\right)  \cap C\left(  \overline{D_{t}}\right)  $ be the solution of the
Dirichlet problem%
\[
\left\{
\begin{array}
[c]{ll}%
\Delta_{p}u_{t}=0 & \text{on }D_{t}\\
u_{t}=1 & \text{on }E_{1}\cap\partial D_{t}\\
u_{t}=0 & \text{on }E_{j}\cap\partial D_{t}\text{, }j\neq1.
\end{array}
\right.
\]
Note that, by the strong maximum principle, $0<u_{t}<1$ in $D_{t}$. Moreover,
as explained in Lemma \ref{lemma_pharmends}, the sequence $\left\{
u_{t}\right\}  _{t\in\mathbb{N}}$ converges, locally uniformly, to a
$p$-harmonic function $u\in C^0(M)  \cap C_{loc}^{1,\alpha}\left(
M\right)  $ satisfying $0\leq u\leq1$. Now, for every $j=1,2$, let $h_{j}$
be the $p$-harmonic function associated to the ends $E_{j}$ costructed in
Lemma \ref{lemma_pharmends}. Recall that $h_{j}$ is the (locally uniform)
limit of the $p$-harmonic function $h_{j,t}$ which satisfy $h_{j,t}=1$ on
$\partial E_{j}$ and $h_{j,t}=0$ on $E_{j}\cap\partial D_{t}$. Define
$k_{1,t}=1-h_{1,t}$. Then, comparing $u_{t}$ and $k_{1,t}$ on $E_{1,t}$ yields
that $u_{t}\geq k_{1,t}$ on $E_{1,t}$. On the other hand, comparing $u_{t}$
and $h_{2,t}$, gives $u_{t}\leq h_{2,t}$ on $E_{2,t}$.
Therefore, taking limits as $t\rightarrow+\infty$, we deduce that $u\geq
h_{1}$ on $E_{1}$ and $u\leq k_{2}$ on $E_{2}$. From this, using
(3) in Lemma \ref{lemma_pharmends}, we conclude that $u$ is non-constant. We
claim that $\left\vert \nabla u\right\vert \in L^{p}(M)  $.
Indeed, for every \thinspace$j=1,...,n$, let $F_{j,t}=E_{j}\backslash E_{j,t}
.$ We think of $u_{t}$ as extended to all
of $M$ by $u_{t}=1$ on $F_{1,t}$ and $u_{t}=0$ on $\cup_{i=2}^{n}F_{i,t}$.
Then, by construction, $u_{t}$ is the equilibrium potential of the condenser
$\left(  F_{1,t},\cup_{i\geq 2}E_{i,t}\cup \Omega\cup E_{1} \right)  $ and we have%
\[
\mathrm{cap}_{p}\left(  F_{1,t},\cup_{i\geq 2}E_{i,t}\cup \Omega\cup E_{1} \right)
=\int_{M}\left\vert \nabla u_{t}\right\vert ^{p}.
\]
On the other hand, take $k_{1,t}$ and extend it to be one on $F_{1,t}$.
Then, $k_{1,t}$ is the equilibrium potential of the
condenser $\left(  F_{1,t},E_{1}\right)  $ and we have%
\[
\mathrm{cap}_{p}\left(  F_{1,t}, E_{1}\right)  =\int_{M}\left\vert
\nabla k_{1,t}\right\vert ^{p}.
\]
By the monotonicity properties of the $p$-capacity, \cite{HKM},
\cite{GT-capacity}, and recalling that $\int_{E_{1,t}}\left\vert \nabla
k_{1,t}\right\vert ^{p}$ is decreasing in $t$, we deduce%
\begin{equation*}
\begin{split}
\int_{M}\left\vert \nabla u_{t}\right\vert ^{p}&=\mathrm{cap}_p\left( F_{1,t}, \cup_{i\geq 2}E_{i,t}\cup \Omega\cup E_{1}\right)  \\
& \leq
\mathrm{cap}_p\left( F_{1,t}, E_1\right)
=
\int_{M}\left\vert \nabla k_{1,t}\right\vert ^{p}=\int_{E_{1,t}}\left\vert \nabla k_{1,t}\right\vert
^{p}\leq C,
\end{split}
\end{equation*}
for some constant $C>0$ independent of $t$. Now observe that, for every domain
$D\subset\subset M$, $\nabla u_{t}\rightarrow\nabla u$ weakly in $L^{p}\left(
D\right)  $ and therefore%
\[
\int_{D}\left\vert \nabla u\right\vert ^{p}\leq\liminf_{t\rightarrow+\infty
}\int_{D}\left\vert \nabla u_{t}\right\vert ^{p}\leq C.
\]
Letting $D\nearrow M$ completes the proof.
\end{proof}

\section{A Liouville-type result for $p$-harmonic functions}

The project of a self-contained proof of Theorem \ref{th_main} will be completed
once we have proved the following Liouville-type result for $p$-harmonic
function with finite $p$-energy.

\begin{theorem}
\label{thm_liouville} Let $\left(  M,g\right)  $ be complete Riemanian
manifold such that $^{M}\Ric\geq-q\left(  x\right)  $ for some continuous
function $q\left(  x\right)  \geq0$. Let $p\geq2$ and assume that the
Schrodinger operator $L_{H}=-\Delta-Hq\left(  x\right)  $ satisfies%
\[
\lambda_{1}^{L_{H}}(M)\geq0
\]
for some $H>p^{2}/4\left(  p-1\right)  $. Then, every $p$-harmonic function
$u:M\rightarrow\mathbb{R}$ of class $C^{1}$ and with finite $p$-energy
$\left\vert \nabla u\right\vert \in L^{p}(M)$ must be constant.
\end{theorem}

Note that the spectral condition is equivalent to the strong positivity of the
operator
\[
-\Delta-\frac{s^{2}}{4\left(  s-1\right)  }q(x)
\]
in the terminology of \cite{CouhlonZhang}.\medskip

In the recent paper \cite{PV-Geom}, the authors obtained a more general result
for manifold-valued $p$-harmonic maps with low regularity. The proof in the
real-valued case of Theorem \ref{thm_liouville} \ appears somewhat more direct.

\begin{proof}
Roughly speaking, the idea is to obtain a Caccioppoli-type inequality for the
energy density $\left\vert \nabla u\right\vert $ of $u$ and this is achieved
by integrating the Bochner formula against suitable test functions.

Note that, by elliptic regularity, $u$ is smooth on the open set%
\[
M_{+}=\left\{  x\in M:\left\vert \nabla u\right\vert \neq0\right\}  .
\]
The standard Bochner formula, which is valid for a generic smooth function,
states that
\[
\frac{1}{2}\Delta\left\vert \nabla u\right\vert ^{2}=\left\vert \Hess\left(
u\right)  \right\vert ^{2}+\left\langle \nabla\Delta u,\nabla u\right\rangle
+ \Ric\left(  \nabla u,\nabla u\right)  ,\text{ on }M_{+}.
\]
Computing the Laplacian on the left hand side, using the Kato inequality
\[
\left\vert \nabla\left\vert \nabla u\right\vert \right\vert ^{2}\leq\left\vert
\Hess (u) \right\vert ^{2}%
\]
and recalling that $\Ric\geq-q\left(  x\right)  $, we deduce
\begin{equation}
\left\vert \nabla u\right\vert \Delta\left\vert \nabla u\right\vert
\geq\left\langle \nabla\Delta u,\nabla u\right\rangle -q\left(  x\right)
\left\vert \nabla u\right\vert ^{2},\text{ on }M_{+}. \label{liouville1}%
\end{equation}
Let $0\leq\rho\in C_{c}^{\infty}\left(  M_{+}\right)  $ be a test function. We
multiply both sides of (\ref{liouville1}) by $\rho^{2}\left\vert \nabla
u\right\vert ^{p-2}$ and we integrate by parts thus obtaining%
\begin{align}
-\int_{M_{+}}\left\langle \nabla\left(  \left\vert \nabla u\right\vert
^{p-1}\rho^{2}\right)  ,\nabla\left\vert \nabla u\right\vert \right\rangle  &
\geq-\int_{M_{+}}\Delta u\operatorname{div}\left(  \rho^{2}\left\vert \nabla
u\right\vert ^{p-2}\nabla u\right) \label{liouville2}\\
&  -\int_{M_{+}}q\left(  x\right)  \rho^{2}\left\vert \nabla u\right\vert
^{p}.\nonumber
\end{align}
We shall take care of each of the integrals in (\ref{liouville2}) separately.

(I) Direct computations and the Cauchy-Schwarz inequality show that%
\begin{align}
-\int_{M_{+}}\left\langle \nabla\left(  \left\vert \nabla u\right\vert
^{p-1}\rho^{2}\right)  ,\nabla\left\vert \nabla u\right\vert \right\rangle  &
\leq2\int_{M_{+}}\rho\left\vert \nabla\rho\right\vert \left\vert \nabla
u\right\vert ^{p-1}\left\vert \nabla\left\vert \nabla u\right\vert \right\vert
\label{liouville3}\\
&  -\left(  p-1\right)  \int_{M_{+}}\rho^{2}\left\vert \nabla u\right\vert
^{p-2}\left\vert \nabla\left\vert \nabla u\right\vert \right\vert
^{2}.\nonumber
\end{align}
Let $\varepsilon>0$ be any small number. Using the elementary inequality
$2ab\leq\varepsilon^{2}a^{2}+\varepsilon^{-2}b^{2}$ we obtain%
\begin{align*}
\int_{M_{+}}2\rho\left\vert \nabla\rho\right\vert \left\vert \nabla
u\right\vert ^{p-1}\left\vert \nabla\left\vert \nabla u\right\vert
\right\vert  &  \leq\varepsilon^{2}\int_{M_{+}}\rho^{2}\left\vert \nabla
u\right\vert ^{p-2}\left\vert \nabla\left\vert \nabla u\right\vert \right\vert
^{2}\\
&  +\varepsilon^{-2}\int_{M_{+}}\left\vert \nabla\rho\right\vert
^{2}\left\vert \nabla u\right\vert ^{p},
\end{align*}
which, inserted into (\ref{liouville3}), yields%
\begin{align}
-\int_{M_{+}}\left\langle \nabla\left(  \left\vert \nabla u\right\vert
^{p-1}\rho^{2}\right)  ,\nabla\left\vert \nabla u\right\vert \right\rangle  &
\leq\varepsilon^{-2}\int_{M_{+}}\left\vert \nabla\rho\right\vert
^{2}\left\vert \nabla u\right\vert ^{p}\label{liouville3a}\\
&  +\left(  \varepsilon^{2}-\left(  p-1\right)  \right)  \int_{M_{+}}\rho
^{2}\left\vert \nabla u\right\vert ^{p-2}\left\vert \nabla\left\vert \nabla
u\right\vert \right\vert ^{2}.\nonumber
\end{align}

(II) Again, by direct computations,%
\begin{align}
-\int_{M_{+}}\Delta u\operatorname{div}\left(  \rho^{2}\left\vert \nabla
u\right\vert ^{p-2}\nabla u\right)   &  =-\int_{M_{+}}\rho^{2}\Delta
u\Delta_{p}u\label{liouville4}\\
&  -2\int_{M_{+}}\rho\Delta u\left\vert \nabla u\right\vert ^{p-2}\left\langle
\nabla\rho,\nabla u\right\rangle .\nonumber
\end{align}
Now, since $u$ is $p$-harmonic, $\Delta_{p}u=0$ and, therefore, the first
summand on the right hand  side  vanishes. On the other hand, expanding the $p$-harmonicity
condition we see that
\[
\Delta u=-\left(  p-2\right)  \left\vert \nabla u\right\vert ^{-1}\left\langle
\nabla\left\vert \nabla u\right\vert ,\nabla u\right\rangle \text{, on }%
M_{+}.
\]
Replacing this expression into (\ref{liouville4}) and manipulating as above,
we conclude%
\begin{align}
\text{LHS}(\ref{liouville4})  &  =2\left(  p-2\right)  \int_{M_{+}}%
\rho\left\vert \nabla u\right\vert ^{-1}\left\langle \nabla\left\vert \nabla
u\right\vert ,\nabla u\right\rangle \left\vert \nabla u\right\vert
^{p-2}\left\langle \nabla\rho,\nabla u\right\rangle \label{liouville5}\\
&  \geq-2\left(  p-2\right)  \int_{M_{+}}\rho\left\vert \nabla\rho\right\vert
\left\vert \nabla u\right\vert ^{p-1}\left\vert \nabla\left\vert \nabla
u\right\vert \right\vert \nonumber\\
&  \geq-\varepsilon^{2}\left(  p-2\right)  \int_{M_{+}}\rho^{2}\left\vert
\nabla u\right\vert ^{p-2}\left\vert \nabla\left\vert \nabla u\right\vert
\right\vert ^{2}-\varepsilon^{-2}\int_{M_{+}}\left\vert \nabla\rho\right\vert
^{2}\left\vert \nabla u\right\vert ^{p}.\nonumber
\end{align}

(III) Recall that, by the spectral assumption,%
\[
\int_{M}\left\vert \nabla\varphi\right\vert ^{2}-Hq\left(  x\right)
\varphi^{2}\geq0,
\]
for every $\varphi\in C_{c}^{\infty}\left(  M\right)  $. Taking $\varphi
=\rho\left\vert \nabla u\right\vert ^{p/2}$ and performing the needed
computations as above, we finally obtain%
\begin{align}
-\int_{M_{+}}q\left(  x\right)  \rho^{2}\left\vert \nabla u\right\vert ^{p}
&  \geq-\left(  H^{-1}+\varepsilon^{-2}H^{-1}p\right)  \int_{M_{+}}\left\vert
\nabla\rho\right\vert ^{2}\left\vert \nabla u\right\vert ^{p}%
\label{liouville6}\\
&  -\left(  \frac{p^{2}}{4}H^{-1}+\varepsilon^{2}H^{-1}p\right)  \int_{M_{+}%
}\rho^{2}\left\vert \nabla u\right\vert ^{p-2}\left\vert \nabla\left\vert
\nabla u\right\vert \right\vert ^{2}.\nonumber
\end{align}

Inserting (\ref{liouville3a}), (\ref{liouville5}) and
(\ref{liouville6}) into (\ref{liouville2}) we  conclude that
\begin{equation}
A\int_{M_{+}}\rho^{2}\left\vert \nabla u\right\vert ^{p-2}\left\vert
\nabla\left\vert \nabla u\right\vert \right\vert ^{2}\leq B\int_{M_{+}%
}\left\vert \nabla\rho\right\vert ^{2}\left\vert \nabla u\right\vert ^{p},
\label{liouville7}%
\end{equation}
where we have set%
\begin{align*}
A  &  =p-1-\frac{p^{2}}{4}H^{-1}-\varepsilon^{2}\left\{   p-1
+H^{-1}p\right\} \\
B  &  =H^{-1}+\varepsilon^{-2}\left\{  H^{-1}p+2\right\}  .
\end{align*}
Note that, by the assumption on $H$, $A>0$ provided $0<\varepsilon<<1$.
Inequality (\ref{liouville7}) is almost the desired Caccioppoli-type
inequality. The main problem to complete the argument, and to deduce the
vanishing of $\left\vert \nabla u\right\vert $ by a standard choice of the
cut-off functions, is that $\rho$ must be supported in $M_{+}$. We need to
extend the validity of (\ref{liouville7}) to any test function compactly
supported in $M$. To this end, we use a  trick introduced by F.
Duzaar and M. Fucks in \cite{DF}. Namely, we define
\[
\varphi_{\delta}=\min\left\{  \frac{|du|^{p/2}}{\delta},1\right\}
\]
for $\delta>0$ and set $\xi=\varphi_{\delta}\eta$ for any $\eta
\in\ C_{c}^{\infty}\left(  M\right)  $. Using the fact that $f\left(
t\right)  =t^{p/2}$ is a Lipshcitz function for $p\geq2$ and that, for a
$p$-harmonic function, $|\nabla u|^{p/2-1}\nabla u\in W_{loc}^{1,2}(M)$
(see e.g. \cite{DF}) it can be verified that $\xi\in W_{0}^{1,2}(M_{+})$. Hence there
exists a sequence $\left\{  \rho_{j}\right\}  _{j=1}^{\infty}\subset
C_{c}^{\infty}(M_{+})$ such that $\rho_{j}\rightarrow\xi$ in $W_{0}^{1,2}(M)$.
Substituting $\rho=\rho_{j}$ into (\ref{liouville7}) and taking the liminf as
$j\rightarrow\infty$, we get%
\begin{align}
A\int_{M_{+}}\eta^{2}(\varphi_{\delta})^{2}\left\vert \nabla u\right\vert
^{p-2}\left\vert \nabla\left\vert \nabla u\right\vert \right\vert ^{2}  &
\leq2B\int_{M_{+}}\eta^{2}\left\vert \nabla\varphi_{\delta}\right\vert
^{2}\left\vert \nabla u\right\vert ^{p}\label{caccioppoli_p2_ve}\\
&  +2B\int_{M_{+}}\left(  \varphi_{\delta}\right)  ^{2}\left\vert
\nabla\eta\right\vert ^{2}\left\vert \nabla u\right\vert ^{p}\nonumber
\end{align}
Finally, we let $\delta\rightarrow0$. Note that $\varphi_{\delta
}\rightarrow1$ pointwise in $M_{+}$. Moreover
\begin{align*}
\int_{M_{+}}\left\vert \nabla u\right\vert ^{p}\left\vert \nabla
\varphi_{\delta}\right\vert ^{2}\eta^{2}  &  =\int_{M_{+}}\left\vert
\nabla u\right\vert ^{2}\frac{\left\vert \nabla\left\vert \nabla u\right\vert
^{p/2}\right\vert ^{2}}{\delta^{2}}\eta^{2}\chi_{\left\{  \left\vert
\nabla u\right\vert ^{p}<\delta^{2}\right\}  }\\
&  \leq\int_{M_{+}}\left\vert \nabla\left\vert \nabla u\right\vert
^{p/2}\right\vert ^{2}\eta^{2}\chi_{\left\{  \left\vert \nabla u\right\vert
^{p}<\delta^{2}\right\}  }%
\end{align*}
and the last term vanishes by dominated convergence as $\delta
\rightarrow0$. Therefore, letting $\delta\rightarrow0$ in
(\ref{caccioppoli_p2_ve}), we finally get the desired Caccioppoli inequality
\begin{equation}
\int_{M_{+}}\eta^{2}\left\vert \nabla u\right\vert ^{p-2}\left\vert
\nabla\left\vert \nabla u\right\vert \right\vert ^{2}\leq C\int_{M_{+}%
}\left\vert \nabla u\right\vert ^{p}\left\vert \nabla\eta\right\vert
^{2},\qquad\forall\eta\in C_{c}^{\infty}(M), \label{caccioppoli}%
\end{equation}
for a suitable constant $C>0$.

As mentioned above, the argument can now be easily completed. By contradiction, suppose $u$ is non-constant. For any fixed
$R>0$, we choose $\eta\left(  x\right)  =\eta_{R}\left(  x\right)  $ so to
satisfy%
\begin{equation}%
\begin{array}
[c]{ll}%
\text{(a) }0\leq\eta\left(  x\right)  \leq1\text{,} & \text{ (b) }\eta\left(
x\right)  =1\text{ on }B_{R}\left(  o\right)  \text{,}\\
\text{(c) }\eta\left(  x\right)  =0\text{ off }B_{2R}\left(  o\right)
\text{,} & \text{(d) }\left\vert \nabla\eta\right\vert \leq2/R\text{ on }M.
\end{array}
\label{cut-off}%
\end{equation}
Whence, we deduce%
\[
\int_{B_{R}\left(  o\right)  \cap M_{+}}\left\vert \nabla u\right\vert
^{p-2}\left\vert \nabla\left\vert \nabla u\right\vert \right\vert ^{2}%
\leq\frac{4C}{R^{2}}\int_{B_{2R}\left(  o\right)  \cap M_{+}}\left\vert \nabla
u\right\vert ^{p},
\]
for some computable positive constant $C$, and letting $R\rightarrow+\infty$
we conclude%
\[
\int_{M_{+}}\left\vert \nabla u\right\vert ^{p-2}\left\vert \nabla\left\vert
\nabla u\right\vert \right\vert ^{2}=0\text{,}%
\]
proving that $\left\vert \nabla u\right\vert =const.$ on every connected
component of $M_{+}$. Since $u$ is non-constant, this implies that $M_{+}=M$
and $|\nabla u|=const.\neq0$. Since, by assumption, $\left\vert
\nabla u\right\vert ^{p}\in L^{1}\left(  M\right)  $, we deduce that
\begin{equation}
\operatorname{vol}M<+\infty. \label{vol}%
\end{equation}
Using this information together with the spectral assumption and choosing
$\eta=\eta_{R}$ to be the cut-off functions defined in (\ref{cut-off}), we get%
\begin{align*}
0  &  \leq\lim_{R\rightarrow+\infty}\int_{B_{2R}\left(  o\right)  }\left\{
H^{-1}\left\vert \nabla\eta\right\vert ^{2}-q\left(  x\right)  \eta
^{2}\right\} \\
&  \leq\lim_{R\rightarrow+\infty}\left\{  \frac{4\;\operatorname{vol}%
B_{2R}\left(  o\right)  }{HR^{2}}-\int_{B_{R}\left(  o\right)  }q\left(
x\right)  \right\} \\
&  =-\int_{M}q\left(  x\right)  \leq0,
\end{align*}
proving that $q\left(  x\right)  =0$, i.e., $^{M}\Ric\geq0$. A well known
result by S.T. Yau and E. Calabi now shows that $M$ has at least a linear
volume growth, contradicting (\ref{vol}).
\end{proof}


\end{document}